\documentclass{amsart}
\usepackage{amsmath,amssymb,amsthm,mathtools, tikz-cd}
\usepackage{hyperref}
\usepackage{algorithm}
\usepackage{algpseudocode}

\usepackage{color}

\usepackage[normalem]{ulem}
\newcommand{\Z}{\mathbb{Z}}
\newcommand{\R}{\mathbb{R}}
\newcommand{\cF}{\mathcal{F}}
\newcommand{\cV}{\mathcal{V}}

\newcommand{\cU}{\mathcal{U}}
\newcommand{\cP}{\mathcal{P}}

\DeclareMathOperator{\GL}{GL}
\DeclareMathOperator{\im}{im}

\DeclareMathOperator{\Span}{span}
\DeclareMathOperator{\SAT}{SAT}
\DeclareMathOperator{\UNSAT}{UNSAT}
\DeclareMathOperator{\supp}{supp}

\newtheorem{theorem}{Theorem}[section]
\newtheorem{proposition}[theorem]{Proposition}
\newtheorem{lemma}[theorem]{Lemma}
\newtheorem{corollary}[theorem]{Corollary}

\theoremstyle{definition}

\theoremstyle{remark}
\newtheorem{remark}[theorem]{Remark}

\title{On the toric lifting properties for simplicial~$3$-spheres}
\date{\today}

\author[S. Choi]{Suyoung Choi}
\address{Department of Mathematics, Ajou University, 206, World cup-ro, Yeongtong-gu, Suwon 16499, Republic of Korea}
\email{schoi@ajou.ac.kr}

\author[Y. Guan]{Yuanxin Guan}
\address{School of Mathematical Sciences, Fudan University, 220 Handan Road, Yangpu District, Shanghai 200433, China}
\email{yxguan21@m.fudan.edu.cn}

\author[H. Jang]{Hyeontae Jang}
\address{School of Mathematics, Korea Institute for Advanced Study, 85 Hoegiro, Dongdaemun-gu, Seoul 02455, Republic of Korea}
\email{hjang1112@kias.re.kr}

\author[Z. L\"u]{Zhi L\"u}
\address{School of Mathematical Sciences, Fudan University, 220 Handan Road, Yangpu District, Shanghai 200433, China}
\email{zlu@fudan.edu.cn}

\thanks{The first author was supported by the National Research Foundation of Korea Grant funded by the Korean Government (RS-2025-00521982). 
The third author was supported by a KIAS Individual Grant (MG105401) at Korea Institute for Advanced Study. 
The second and fourth authors were partially supported by the grant SIMIS-ID-2025-TP}

\begin{document}
\begin{abstract}
    We study the lifting problem for mod~$2$ characteristic maps over simplicial $3$-spheres. 
    Using a bad-block partition of the universal complex~$X(\Z_2^4)$, we prove an avoidance criterion for liftability. 
    We show that every simplicial $3$-sphere with at most $20$ vertices has the toric lifting property. 
    We also obtain image-size and join-type results, and prove sharpness of the image-size bound in the universal-complex sense.
\end{abstract}

\maketitle

\section{Introduction}

Characteristic maps are among the basic combinatorial objects in toric topology.
An integral characteristic map encodes the local integral unimodularity data that appears in quasitoric and toric-type constructions, while a mod~$2$ characteristic map is the corresponding real or small-cover analogue.
The \emph{lifting problem} asks when a mod~$2$ characteristic map is the mod~$2$  reduction of an integral characteristic map.
Thus the problem gives a natural way to compare characteristic data over $\Z_2$ with characteristic data over~$\Z$.

Let $K$ be an $(n-1)$-dimensional simplicial sphere on $[m]=\{1,\ldots,m\}$, and let
$$
    \lambda^\R\colon [m]\longrightarrow \Z_2^n
$$
be a mod~$2$ characteristic map, that is, the images of the vertices of every facet of~$K$ form a basis of~$\Z_2^n$.
The lifting problem, originally proposed by Zhi L\"u at the Problem Session of the conference ``Toric
Topology 2011'' in Osaka, as documented in \cite{Choi-Park2016Wedge}, asks whether there exists an integral characteristic map
$$
    \lambda\colon [m]\longrightarrow \Z^n
$$
that is, a map whose values on the vertices of every facet of $K$ form a basis of $\Z^n$,
such that the following diagram commutes:
$$
    \begin{tikzcd}
        & \Z^n \arrow[d, "\bmod 2"] \\
        \left[m \right] \arrow[ur, "\exists \lambda", dashed] \arrow[r, "\lambda^{\R}"'] & \Z_2^n.
    \end{tikzcd}
$$
In this case, we call $\lambda$ a \emph{lift} of $\lambda^\R$.
If a lift $\lambda$ sends $[m]$ to $\{0,1\}$-vectors, then it is called the \emph{$\{0,1\}$-lift} of~$\lambda^\R$.
We say that $K$ has the \emph{toric lifting property} if every mod~$2$ characteristic map over $K$ admits a lift.

We now elaborate the fundamental obstruction behind the lifting problem.
The primary obstacle to solving the lifting problem does not lie in separately lifting each vector of the mod~$2$ characteristic map.
Indeed, every vector in $\Z_2^n$ admits infinitely many integral representatives in $\Z^n$, all congruent to one another modulo~$2$.
Instead, the key difficulty is to choose these representatives simultaneously so that the images of the vertices of every facet form a unimodular basis of~$\Z^n$.
Equivalently, the problem measures the gap between nonsingularity over $\Z_2$ and unimodularity over $\Z$.
Although the determinant of an integer matrix which is nonzero mod~$2$ is merely odd over~$\Z$, it need not be equal to $\pm1$. 
For example, a 0-1 matrix with odd determinant may have determinant $\pm3$, which fails the unimodularity condition required for integral characteristic maps; such matrices form the so-called bad blocks. 
We formalize this obstruction via the bad-block partition of the universal complex~$X(\Z_2^4)$ constructed in Section~\ref{universal cpx}.
The lifting problem asks whether this obstruction can be  eliminated globally by a suitable choice of integral representatives. 

Several positive results have been obtained on the toric lifting problem. 
First, every simplicial~$2$-sphere has the toric lifting property.
More generally, Choi, Jang, and Vall\'ee~\cite{Choi-Jang-Vallee2025JLMS} proved that every $(n-1)$-dimensional PL~sphere with at most $n+4$ vertices has the toric lifting property. 
For $n=4$, their theorem covers PL~$3$-spheres with at most $8$ vertices.
It is worth noting that all $(n-1)$-dimensional PL spheres form a proper subclass of $(n-1)$-dimensional simplicial spheres when $n\geq 6$, while the two families coincide when $n\leq 4$.  
The case $n=5$ remains an open problem.

This paper focuses on simplicial~$3$-spheres, the first substantially nontrivial family.
We first establish a bad-block avoidance criterion from $X(\Z_2^4)$, formulated as Lemma~\ref{lem:block-avoidance}, which serves as our foundational technical tool. 
Based on this criterion, we obtain the following results.
\begin{enumerate}
    \item Theorem~\ref{thm:main} shows that every simplicial~$3$-sphere with at most $20$ vertices has the toric lifting property. 
        This extends the previous result for PL~$3$-spheres with at most $8$ vertices.
    \item Corollary~\ref{cor:neighborly} establishes the toric lifting property for all neighborly simplicial~$3$-spheres. 
        Theorem~\ref{thm:flag25} gives the same conclusion for flag simplicial~$3$-spheres with at most $25$ vertices.
    \item Theorem~\ref{thm:image-size-13} shows that every mod~$2$ characteristic map with at most $13$ distinct vertex images over a simplicial~$3$-sphere admits a lift. 
        Corollary~\ref{cor:orientable-small-cover} then gives the same conclusion for every orientable mod~$2$ characteristic map.
    \item Theorem~\ref{thm:join-type-3-spheres} establishes the toric lifting property for every simplicial~$3$-sphere that is a nontrivial join of two lower-dimensional simplicial spheres.
    \item Proposition~\ref{prop:14-subset-nonliftable} shows that, for every $14$-element subset $S$ of the vertex set of $X(\Z_2^4)$, the inclusion~$X[S]\hookrightarrow X(\Z_2^4)$ does not admit a lift. 
        Thus, the bound in Theorem~\ref{thm:image-size-13} is sharp in the universal-complex sense. 
        We also construct a simplicial~$3$-sphere admitting a nondegenerate simplicial map to $X(\Z_2^4)$ that is surjective on facets.
\end{enumerate}

\section{Universal Complexes} \label{universal cpx}
Let $A=(a_1,\dots,a_{15})$ be the $4\times 15$ $\Z_2$-matrix whose columns are the nonzero vectors of $\Z_2^4$ listed in lexicographic order
$$
    A= \left(
    \begin{array}{ccccccccccccccc}
        0&0&0&0&0&0&0&1&1&1&1&1&1&1&1\\
        0&0&0&1&1&1&1&0&0&0&0&1&1&1&1\\
        0&1&1&0&0&1&1&0&0&1&1&0&0&1&1\\
        1&0&1&0&1&0&1&0&1&0&1&0&1&0&1
    \end{array} \right),
$$
and let $X\bigl(\Z_2^4\bigr)$ be the corresponding universal complex formed by all linear independent subsets of $\Z_2^4$.
Thus the vertex set of~$X(\Z_2^4)$ is $\cV\coloneq\cV(X(\Z_2^4)) = \{a_1,\dots,a_{15}\}$, and its facet set is
$$
    \cF\coloneq\cF(X(\Z_2^4)) =\Bigl\{ \{a_{i_1},a_{i_2},a_{i_3},a_{i_4}\}\subset \cV \Bigm| \det(a_{i_1},a_{i_2},a_{i_3},a_{i_4})\equiv 1 (\bmod 2) \Bigr\}.
$$
Then $|\cF|=\frac{|\GL(4,\Z_2)|}{4!}=\frac{(2^4-1)(2^4-2)(2^4-2^2)(2^4-2^3)}{4!}=840$.

With our choice of ordering, and viewing each $a_i$ as an integer vector in $\{0,1\}^4\subset \Z^4$, the set of facets of $X(\Z_2^4)$ whose determinants have absolute value $3$ is
$$
    B= \left\{
        \begin{aligned}
            &\{a_3,a_5,a_9,a_{14}\},\ \{a_3,a_6,a_{10},a_{13}\},\ \{a_5,a_6,a_{11},a_{12}\},\\
            &\{a_7,a_9,a_{10},a_{12}\},\ \{a_7,a_{11},a_{13},a_{14}\}
        \end{aligned}
    \right\}.
$$
We call $B$ the \emph{standard bad block}.

Note that the action of $\GL(4,\Z_2)$ on $X(\Z_2^4)$ permutes the facets of $X(\Z_2^4)$.
One interesting fact, originally established in the proof of \cite[Theorem~6.1]{Choi-Jang-Vallee2025JLMS}, is that $\cF$ admits a partition
$$
    \cF=\bigsqcup_{\alpha=1}^{168} B_\alpha
$$
into $168$ subsets of size $5$, and for each $\alpha$ there exists $g_\alpha\in \GL(4,\Z_2)$ such that
$$
    g_\alpha(B_\alpha)=B.
$$
We may assume that $B_1=B$.

For a facet $\tau=\{u_1,u_2,u_3,u_4\}\in \cF$, set
$$
    P(\tau)\coloneq \{\tau\}\cup \bigl\{ \{u_i\}\cup \{u_j+u_i\mid j\neq i\} \bigm| i=1,2,3,4 \bigr\}.
$$
With this notation, the standard bad block is
$$
    B=P(\{a_3,a_5,a_9,a_{14}\}).
$$
The distinct sets $P(\tau)$, as $\tau$ runs over the facets of $X(\Z_2^4)$, form the bad-block partition.
In particular, for every basis $\{u,p,q,r\}$ of $\Z_2^4$, the five facets
\begin{align*}
    &\{u,p,q,r\},\quad \{u,p+u,q+u,r+u\},\quad  \{u+p,p,q+p,r+p\}, \\
    &\{u+q,p+q,q,r+q\},\quad \{u+r,p+r,q+r,r\}  
\end{align*}
belong to the same block of the bad-block partition.

The following simple observation explains why bad blocks control the validity of integral lifts.

\begin{lemma}\label{lem:avoid-standard-bad-block}
    Let $Y$ be a pure $3$-dimensional simplicial complex, and let $\eta^\R$ be a mod~$2$ characteristic map over~$Y$.
    Suppose that no facet image of $\eta^\R$ lies in the standard bad block~$B$.
    Then the map obtained by taking the $\{0,1\}$-representative of each vector $\eta^\R(v)\in \Z_2^4$ is a $\{0,1\}$-lift of $\eta^\R$.
\end{lemma}
\begin{proof}
    Let $\eta$ be the $\{0,1\}$-representative lift of $\eta^\R$.
    For each facet $\sigma$ of $Y$, the image~$\widehat{\eta^\R}(\sigma)$ is a facet of $X(\Z_2^4)$.
    The determinant of the corresponding $0$-$1$ matrix is odd.
    Since every $4\times4$ 0-1 matrix has determinant of absolute value at most~$3$, an odd determinant is equal to~$\pm 1$ or~$\pm 3$.
    By the definition of $B$, the only facets of $X(\Z_2^4)$ whose $0$-$1$ determinant has absolute value $3$ are precisely the facets in~$B$.
    Since $\widehat{\eta^\R}(\sigma)\notin B$, this determinant has absolute value $1$.
    Thus the images of the vertices of every facet of $Y$ form a unimodular basis of $\Z^4$, and so $\eta$ is a $\{0,1\}$-lift.
\end{proof}

We note that a pure $3$-dimensional simplicial complex $Y$ supports a characteristic map $\lambda^\R$ if and only if there is a nondegenerate simplicial map $\widehat{\lambda^\R} \colon Y \to X(\Z_2^4)$.
For a facet $\sigma = \{i_1, i_2, i_3, i_4\}$ of $Y$, we write
$$
    \widehat{\lambda^\R}(\sigma) \coloneq\{\widehat{\lambda^\R}(i_1),\widehat{\lambda^\R}(i_2),\widehat{\lambda^\R}(i_3),\widehat{\lambda^\R}(i_4)\}\in\cF,
$$
and define
$$
    \Sigma(\lambda^\R)\coloneq \Bigl\{\widehat{\lambda^\R}(\sigma)\ \Bigm|\ \sigma\in\cF(Y)\Bigr\}.
$$

\begin{lemma}[Bad-block avoidance criterion]\label{lem:block-avoidance}
    Let $Y$ be a pure $3$-dimensional simplicial complex, and let $\lambda^\R$ be a mod~$2$ characteristic map over $Y$.
    If $\Sigma(\lambda^\R)$ is disjoint from some block~$B_\alpha$ of the bad-block partition
    $$
        \cF(X(\Z_2^4))=\bigsqcup_{\alpha=1}^{168}B_\alpha,
    $$
    then $\lambda^\R$ admits a lift.
    Or equivalently, if $\Sigma(\lambda^\R)$ meets fewer than $168$ blocks of this partition, then $\lambda^\R$ admits a lift.
\end{lemma}
\begin{proof}
    Choose $\alpha$ such that $\Sigma(\lambda^\R)\cap B_\alpha=\varnothing$.
    Since $g_\alpha(B_\alpha)=B$, we have
    $$
        \Sigma(g_\alpha\lambda^\R)\cap B=\varnothing.
    $$
    Let $\widetilde\lambda$ be the map obtained by taking the $\{0,1\}$-representative of each vector~$g_\alpha\lambda^\R(v)$.
    By Lemma~\ref{lem:avoid-standard-bad-block}, $\widetilde\lambda$ is a $\{0,1\}$-lift of $g_\alpha\lambda^\R$.
    
    The reduction map $\GL(4,\Z) \to \GL(4,\Z_2)$ is surjective, since $\GL(4,\Z_2)$ is generated by elementary matrices, each of which has an integral unimodular lift.
    Choose an integral lift $G_\alpha\in \GL(4,\Z)$ of $g_\alpha$, and define
    $$
        \lambda\coloneq G_\alpha^{-1}\widetilde\lambda.
    $$
    Then $\lambda$ is an integral characteristic map and
    $$
        \lambda=G_\alpha^{-1}\widetilde\lambda \equiv g_\alpha^{-1}(g_\alpha\lambda^\R)=\lambda^\R \pmod 2.
    $$
    Hence $\lambda$ is a lift of $\lambda^\R$.
\end{proof}

\begin{proposition}\label{prop:168}
    Let $Y$ be a pure $3$-dimensional simplicial complex, and let $\lambda^\R$ be a mod~$2$ characteristic map over $Y$.
    If $|\Sigma(\lambda^\R)|<168$, then $\lambda^\R$ admits a lift.
    In particular, if $K$ is a simplicial $3$-sphere with $f_3(K)<168$, then any mod~$2$ characteristic map over $K$ admits a lift.
\end{proposition}
\begin{proof}
    Since the bad-block partition has $168$ blocks and $|\Sigma(\lambda^\R)|<168$, the set~$\Sigma(\lambda^\R)$ is disjoint from at least one block.
    The result follows from Lemma~\ref{lem:block-avoidance}.
    The final assertion follows because $|\Sigma(\lambda^\R)|\le f_3(K)$.
\end{proof}

\section{Toric lifting properties for 3-spheres} \label{toric lifting property}

Based upon the bad-block avoidance criterion established in Section~\ref{universal cpx}, we first prove our  vertex bound theorem for general simplicial 3-spheres and derive a corollary for neighborly  simplicial 3-spheres.
\begin{theorem}\label{thm:main}
    Let $K$ be a simplicial $3$-sphere with at most $20$ vertices.
    Then $K$ has the toric lifting property.
\end{theorem}
\begin{proof}
    Let $K$ be a simplicial $3$-sphere on $[m]$ with $m \leq 20$.
    Then, by the \emph{upper bound theorem}~\cite{Stanley1975}, the number $f_3(K)$ of facets is at most $\frac{m(m-3)}{2}$.
    In particular, if $m \leq 19$, then $f_3(K) \leq \frac{19 \cdot 16}{2} = 152$.
    Now suppose that $m=20$ and $K$ supports a mod~$2$ characteristic map, and let $\widehat{\lambda^\R} \colon K\to X\bigl(\Z_2^4\bigr)$ be the corresponding map.
    Since $\widehat{\lambda^\R}$ is nondegenerate, the endpoints of every edge of $K$ must be sent to distinct vertices of $X(\Z_2^4)$, and, hence, the $1$-skeleton of $K$ is $15$-colorable.
    
    Let $n_1,\dots,n_{15}$ be the color class sizes. 
    Then $n_1+\cdots+n_{15}=20$, and the number of edges satisfies
    $$
        f_1(K) \leq \binom{20}{2}-\sum_{i=1}^{15}\binom{n_i}{2}.
    $$
    Since $20$ vertices are distributed among $15$ classes, at least five pairs of vertices must lie in the same color class, and therefore
    $$
        \sum_{i=1}^{15}\binom{n_i}{2}\ge 5.
    $$
    Thus $f_1(K)\le \binom{20}{2}-5=185$.
    
    Now $K$ is a simplicial $3$-sphere, so every triangle lies in exactly two tetrahedra. 
    Hence $2f_2(K)=4f_3(K)$, that is, $f_2(K)=2f_3(K)$.
    Applying Euler's relation, we obtain
    $$
        20-f_1(K)+2f_3(K)-f_3(K)=0,
    $$
    and, hence, $f_3(K)=f_1(K)-20\le 185-20=165$.
    
    Hence, by Proposition~\ref{prop:168}, $K$ has the toric lifting property.
\end{proof}

A simplicial complex $K$ is called \emph{neighborly} if every pair of distinct vertices of~$K$ spans an edge of~$K$.

\begin{corollary}\label{cor:neighborly}
    Every neighborly simplicial $3$-sphere has the toric lifting property.  
\end{corollary}
\begin{proof} 
    Let $K$ be a neighborly simplicial $3$-sphere on $[m]$.
    Let $\lambda^\R$ be an arbitrary mod~$2$ characteristic map over~$K$, and let $\widehat{\lambda^{\R}}\colon K\to X(\Z_2^4)$ be the associated nondegenerate simplicial map.
    Since $K$ is neighborly, $\widetilde{\lambda^{\R}}$ is injective on vertices.
    Hence $m\leq 15$, and Theorem~\ref{thm:main} implies that $\lambda^\R$ admits a lift.
\end{proof}
\begin{remark} 
    The lifting problem for neighborly simplicial~$3$-spheres was studied in \cite{Baralic-Milenkovic2022}.
    That work treats the polytopal case and verifies the lifting conjecture for neighborly simple $4$-polytopes with at most $12$ facets through a computer-assisted classification. 
    By contrast, Corollary~\ref{cor:neighborly} gives a uniform proof for all neighborly simplicial~$3$-spheres.
\end{remark}

Using Zheng’s upper bound theorem on flag 3-manifolds~\cite{Zheng2017}, we obtain a weaker restriction on the vertex bound for flag simplicial 3-spheres.
A simplicial complex is called \emph{flag} if any set of vertices of $K$ which are pairwise connected by edges spans a simplex of~$K$.

\begin{theorem}\label{thm:flag25}
    Let $K$ be a flag simplicial $3$-sphere on $[m]$ with $m\le 25$.
    Then $K$ has the toric lifting property.
\end{theorem}
\begin{proof}
    Since $K$ is a flag simplicial $3$-sphere, it is in particular a flag $3$-manifold.
    By the flag upper bound theorem of Zheng~\cite{Zheng2017}, among all flag $3$-manifolds on $m$ vertices, we have
    $$
        f_1(K)\leq \Bigl\lfloor \frac{m^2}{4}\Bigr\rfloor + m.
    $$
    Since $K$ is a simplicial $3$-sphere, we have $f_3(K)=f_1(K)-f_0(K)=f_1(K)-m$.
    Therefore
    $$
        f_3(K)\le \Bigl\lfloor \frac{m^2}{4}\Bigr\rfloor.
    $$
    If $m\le 25$, then
    $$
        f_3(K)\le \Bigl\lfloor \frac{25^2}{4}\Bigr\rfloor = 156 < 168.
    $$
    Hence Proposition~\ref{prop:168} applies, and every mod~$2$ characteristic map on $K$ admits a lift.
\end{proof}

We now study liftability controlled by the size of the vertex image of the characteristic map.
For a subset $S\subset \cV$, we write $X[S]\coloneq\{\sigma\in X(\Z_2^4)\mid \sigma\subset S\}$ for the induced subcomplex on $S$.
Let $K$ be a simplicial $3$-sphere.
Recall that there is a mod~$2$ characteristic map $\lambda^\R$ over $K$ if and only if there is a nondegenerate simplicial map $\widehat{\lambda^\R}\colon K \to X(\Z_2^4)$.
Moreover, $\lambda^\R$ has a lift whenever the inclusion~$X[ \im \lambda^\R] \hookrightarrow X(\Z_2^4)$ admits a lift 
$$
    \begin{tikzcd}
        &  & X(\Z^4) \arrow[d, "\bmod 2"] \\
        K \arrow[r, "\widehat{\lambda^{\R}}"'] & X[ \im \lambda^\R] \arrow[ur, dashed] \arrow[r, hook] & X (\Z_2^4),
    \end{tikzcd}
$$
where $X(\Z^4)$ denotes the universal simplicial complex of unimodular subsets of primitive vectors in $\Z^4$.

\begin{lemma}\label{lem:qle13}
    If $S\subset \cV$ satisfies $|S|\leq 13$, then the inclusion $X[S]\hookrightarrow X(\Z_2^4)$ admits a lift to a nondegenerate simplicial map
    $$
        X[S]\longrightarrow X(\Z^4).
    $$
\end{lemma}
\begin{proof}
    Up to the action of $\GL(4,\Z_2)$, there is exactly one orbit of $13$-element subsets of $\cV$, represented by
    $$
        S=\cV\setminus\{a_5,a_{14}\}.
    $$
    Order the vertices of $S$ as
    $$
        (a_1,a_2,a_3,a_4,a_6,a_7,a_8,a_9,a_{10},a_{11},a_{12},a_{13},a_{15}),
    $$
    and consider the integer matrix
    $$
        L=\left(
            \begin{array}{ccccccccccccc}
                0&0&0&0&0&0&1&-1&-1&-1&-1&-1&-1\\
                0&0&0&1&-1&-1&0&0&0&0&-1&-1&-1\\
                0&1&-1&0&-1&-1&0&0&1&-1&0&0&-1\\
                1&0&-1&0&0&-1&0&-1&0&-1&0&-1&-1
            \end{array}\right).
    $$
    Each column reduces mod~$2$ to the corresponding vertex of $S$, and a direct determinant computation shows that if $\{a_{i_1}, a_{i_2}, a_{i_3}, a_{i_4}\}\in \cF(X[S])$, the matrix formed by the four corresponding columns of~$L$ has determinant~$\pm1$. 
    Hence $L$ defines a lift of the inclusion~$X[S] \hookrightarrow X(\Z_2^4)$.

    Now let $S^\prime\subset \cV$ with $|S^\prime|\le 13$.
    Choose a $13$-element subset $S^{\prime\prime}$ such that $S^\prime\subset S^{\prime\prime}$.
    Since all $13$-element subsets lie in the same $\GL(4,\Z_2)$-orbit, there exists $g\in \GL(4,\Z_2)$ such that
    $$
        g(S^{\prime\prime})=S.
    $$
    Then $g(S^\prime)\subset S$, so the lift of $X[S]$ restricts to a lift of $X[g(S^\prime)]$.
    
    Choose an integral lift $G \in \GL(4,\Z)$ of $g$.
    Composing with $G^{-1}$, we obtain a lift of $X[S^\prime] \hookrightarrow X(\Z_2^4)$.
\end{proof}

The following theorem is an immediate consequence of Lemma~\ref{lem:qle13}.
\begin{theorem}\label{thm:image-size-13}
    Let $K$ be a simplicial $3$-sphere, and let $\lambda^\R$ be a mod~$2$ characteristic map over $K$.
    If $|\im \lambda^\R| \leq 13$, then $\lambda^\R$ has a lift.
\end{theorem}
\begin{proof}
    The corresponding nondegenerate simplicial map factors through the induced subcomplex $X[\im \lambda^\R]$.
    Since $|\im \lambda^\R|\leq 13$, Lemma~\ref{lem:qle13} gives a lift of the inclusion $X[\im \lambda^\R]\hookrightarrow X(\Z_2^4)$.
    Composing with this lift gives a lift of $\lambda^\R$.
\end{proof}

Let $K$ be a simplicial $3$-sphere.
A mod~$2$ characteristic map~$\lambda^\R$ is called \emph{orientable} if, after applying a suitable element of
$\GL(4,\Z_2)$, every value of $\lambda^{\R}$ has odd coordinate sum.
It is known from~\cite{Nakayama-Nishimura2005} that an orientable characteristic map corresponds to an orientable small cover.
 
\begin{corollary}\label{cor:orientable-small-cover}
    Every orientable mod~$2$ characteristic map over a simplicial $3$-sphere admits a lift.
\end{corollary}
\begin{proof}
    After applying a suitable element of $\GL(4,\Z_2)$, we may assume that
    $$
        \im(\lambda^{\R})\subset \{x\in (\Z_2)^4 \mid x_1+x_2+x_3+x_4=1\}.
    $$
    Hence $|\im(\lambda^{\R})|\le 8$, so Theorem~\ref{thm:image-size-13} gives a lift after this basis change.
    Composing with an integral lift of the inverse basis change gives a lift of the original $\lambda^\R$.
\end{proof}

\section{Join complexes} \label{join complex}

We next apply our bad-block obstruction criterion to simplicial 3-spheres that decompose as simplicial joins, establishing universal liftability for all such joined 3-sphere constructions.

Let $K_1$ and $K_2$ be simplicial complexes on $\cV_1$ and $\cV_2$, respectively. The \emph{join} of $K_1$ and $K_2$ is the simplicial complex 
	$$
	    K_1*K_2 \coloneq \{\tau\subset \cV_1\sqcup \cV_2 \mid \tau = \tau_1\cup \tau_2, \tau_1\in K_1, \tau_2\in K_2\}
	$$
on the set $\cV_1\sqcup \cV_2$.

\begin{lemma}\label{lem:join-1-1}
    Let $K$ and $L$ be pure $1$-dimensional simplicial complexes.
    Then every mod~$2$ characteristic map over $K\ast L$ admits a lift.
\end{lemma}
\begin{proof}
    Let $\lambda^\R$ be a mod~$2$ characteristic map over $K\ast L$.
    A facet of $K\ast L$ is of the form $\tau_1\cup\tau_2$, where $\tau_1$ is an edge of $K$ and $\tau_2$ is an edge of $L$.
    Set
    $$
        S\coloneq\{\Span(\lambda^\R(\tau_1))\mid \tau_1\in\cF(K)\},\quad \text{ and } \quad T\coloneq\{\Span(\lambda^\R(\tau_2))\mid \tau_2\in\cF(L)\}.
    $$
    Each element of $S$ and $T$ is a $2$-dimensional subspace of $\Z_2^4$.
    Moreover, if $V\in S$ and $W\in T$, then $V\cap W=0$, because the union of the corresponding two edges is a facet of $K\ast L$.

    For a family $\cU$ of $2$-dimensional subspaces of $\Z_2^4$, write
    $$
        \overline{\cU}\coloneq\{W\le \Z_2^4\mid \dim W=2,\ V\cap W=0\text{ for all }V\in\cU\}.
    $$
    We claim that if $|\cU|\ge 1$, then $|\overline{\cU}|\le16$; if $|\cU|\ge2$, then $|\overline{\cU}|\le8$; and if $|\cU|\ge3$, then $|\overline{\cU}|\le4$.

    To prove the claim, choose $V_1\in\cU$ and set $U\coloneq \Z_2^4/V_1$.
    A $2$-dimensional subspace~$W$ satisfies $V_1\cap W=0$ if and only if the quotient projection $\pi\colon\Z_2^4\to U$ restricts to an isomorphism $W\to U$.
    Equivalently, $W$ is the image of a unique section $s_W\colon U\to\Z_2^4$ of $\pi$.
    Since $\dim U=2$ and each coset of $V_1$ has four elements, there are at most $4\cdot4=16$ such sections.
    Hence $|\overline{\cU}|\le16$ when $|\cU|\ge1$.

    Now suppose $|\cU|\ge2$, and choose $V_2\in\cU$ with $V_2\ne V_1$.
    If $V_1\cap V_2=0$, then $V_2$ is the image of another section $s_2\colon U\to\Z_2^4$.
    The condition $V_2\cap W=0$ is equivalent to saying that $s_W+s_2\colon U\to V_1$ is injective, hence an isomorphism.
    Thus there are at most $|\GL(2,\Z_2)|=6$ possibilities for $s_W + s_2$, and also for $W$.
    If $V_1\cap V_2\ne0$, choose $u_2\in V_2\setminus V_1$ and put $\bar u_2\coloneq\pi(u_2)$.
    Then $\pi(V_2)=\{0,\bar u_2\}$, and the coset $V_1+u_2$ contains exactly two vectors of $V_2$.
    The condition $V_2\cap W=0$ is equivalent to $s_W(\bar u_2)\notin V_2$.
    There are two possible choices for $s_W(\bar u_2)$ and four choices for the value of $s_W$ on a complementary basis vector of $U$.
    Hence there are at most $2\cdot4=8$ possibilities for $W$.
    Therefore $|\overline{\cU}|\le8$ when $|\cU|\ge2$.

    Finally suppose $|\cU|\ge3$, and choose distinct elements $V_1,V_2,V_3\in\cU$.
    If $V_1\cap V_2=0$, then, as above, $s_W+s_2\colon U\to V_1$ must be an isomorphism.
    Since $V_3$ is distinct from both $V_1$ and $V_2$, $|(V_1\cap V_3)\cup (V_2\cap V_3)| = |V_1\cap V_3| + |V_2\cap V_3| - |V_1\cap V_2\cap V_3|\leq 2+2-1 = 3$. But $|V_3| = 4$, so there is a vector $z\in V_3$ but $z\notin V_1\cup V_2$. Then the components of $z$ with respect to $\Z_2^4=V_1\oplus V_2$ are both nonzero.
    Write $z=z_1+z_2$ with $0\ne z_1\in V_1$ and $0\ne z_2\in V_2$.
    Since $V_3\cap W=0$, $s_W+s_2$ cannot send $\pi(z)$ to~$z_1$.
    Exactly two elements of $\GL(2,\Z_2)$ send the nonzero vector $\pi(z)$ to $z_1$, so there are at most $6-2=4$ possibilities for $W$.

    It remains to consider the case in which $V_i\cap V_j\ne0$ for all distinct $i,j\in\{1,2,3\}$.
    Choose $u_i\in V_i\setminus V_1$ and put $\bar u_i\coloneq\pi(u_i)$ for $i=2,3$.
    As above, the conditions $V_i\cap W=0$ are equivalent to $s_W(\bar u_i)\notin V_i$ for $i=2,3$.
    If $\bar u_2\ne\bar u_3$, then the two vectors $\bar u_2,\bar u_3$ form a basis of $U$, and each condition leaves at most two choices.
    Thus there are at most $2\cdot2=4$ sections.
    If $\bar u_2=\bar u_3$, then $s_W(\bar u_2)$ must avoid both~$V_2$ and $V_3$ inside the four-element coset $V_1+u_2$.
    The two intersections $(V_1+u_2)\cap V_2$ and $(V_1+u_2)\cap V_3$ are distinct two-element subsets, and hence their union has at least three elements.
    Thus there is at most one choice for $s_W(\bar u_2)$, and then at most four choices for the value on a complementary basis vector.
    Again there are at most four sections.
    This proves the claim.

    Obviously, $T\subset\overline S$ and $S\subset\overline T$.
    Applying the claim to $T$, since $T$ is nonempty, we have $|S|\le16$.
    Hence $|S||T|\le16$: indeed, if $|S|=1$, then $|T|\le16$; if $|S|=2$, then $|T|\le8$; if $3\le |S|\le4$, then $|T|\le4$; if $5\le |S|\le8$, then $|T|\le2$ by symmetry; and if $9\le |S|\le16$, then $|T|\le1$ by symmetry.

    Each $2$-dimensional vector space over $\Z_2$ has exactly three unordered bases.
    Therefore, for each pair $(V,W)\in S\times T$, there are at most $3\cdot3=9$ possible facet images of $K\ast L$ with edge spans $V$ and $W$.
    Consequently,
    $$
        |\Sigma(\lambda^\R)|\le 9|S||T|\le 144<168.
    $$
    Proposition~\ref{prop:168} implies that $\lambda^\R$ admits a lift.
\end{proof}

\begin{lemma}\label{lem:join-0-2}
    Let $K$ be a nonempty $0$-dimensional simplicial complex, and let $L$ be a pure $2$-dimensional simplicial complex.
    Then every mod~$2$ characteristic map over $K\ast L$ admits a lift.
\end{lemma}

\begin{proof}
    Let $\lambda^\R$ be a mod~$2$ characteristic map over $K\ast L$.
    Denote the vertex set of~$K$ by $[m]$.
    The facets of $K\ast L$ are precisely the sets $\{i\}\cup\sigma$, where $i\in[m]$ and $\sigma\in\cF(L)$.
    Set $v_i\coloneq\lambda^\R(i)$ for $i\in[m]$ and 
    $$
        V\coloneq\Span\{v_i+v_1\mid i\in[m]\}.
    $$
    For each facet $\sigma$ of $L$, set
    $$
        H_\sigma\coloneq\Span(\lambda^\R(\sigma)).
    $$
    By the nonsingularity condition, $H_\sigma$ is a hyperplane of $\Z_2^4$ and $v_i\notin H_\sigma$ for every~$i\in[m]$.
    Since $\Z_2^4/H_\sigma$ is $1$-dimensional, we have $v_i+v_j\in H_\sigma$ for all $i,j\in[m]$.
    Hence $V\subset H_\sigma$ for every facet $\sigma$ of $L$.

    Let $d\coloneq\dim V$.
    Since $v_1\notin H_\sigma$ and $V\subset H_\sigma$, we have $v_1\notin V$.
    After applying an element of $\GL(4,\Z_2)$, we may assume that $v_1=e_1$, and $V=\Span\{e_2,\ldots,e_{d+1}\}$, where $\{e_1,e_2,e_3,e_4\}$ is the standard basis of $\Z_2^4$.
    Then $v_i\in V+e_1$ for all $i\in[m]$.

    Thus every facet image of $K\ast L$ is of the form $\{u,p,q,r\}$, where $u\in V+e_1$ and $\{p,q,r\}$ is a basis of a hyperplane $H$ satisfying $V\subset H$ and $e_1\notin H$.
    
    We now show, case by case, that all possible facet images of this form meet fewer than $168$ blocks of the bad-block partition.

    \noindent\textsf{\textbf{Case 1: $d=0$.}}
    In this case $V=0$, so $u=e_1$.
    There are eight hyperplanes $H$ not containing $e_1$, and each such hyperplane has $28$ unordered bases.
    Hence there are $8\cdot28=224$ possible facets of the form $\{e_1,p,q,r\}$.
    For every such facet, the two facets $\{e_1,p,q,r\}$ and $\{e_1,p+e_1,q+e_1,r+e_1\}$ belong to the same bad block, and the second one is again of the same possible form.
    Therefore these possible facets meet at most
    $$
        \frac{8\cdot28}{2}=112<168
    $$
    blocks.

    \noindent\textsf{\textbf{Case 2: $d=1$.}}
    In this case $V=\Span\{e_2\}$.
    There are four hyperplanes $H$ containing~$V$ and not containing $e_1$.
    Fix such an $H$.
    A direct count in the $3$-dimensional space~$H$ shows that the numbers of unordered bases $B_H=\{p,q,r\}$ satisfying 
    $$
        e_2\in B_H, \qquad e_2\text{ is the sum of two vectors in }B_H, \qquad e_2=p+q+r
    $$
    are $12$, $12$, and $4$, respectively.

    If $e_2\in B_H$, say $p=e_2$, then for every $u\in V+e_1$ the two facets $\{u,e_2,q,r\}$ and $\{u+e_2,e_2,q+e_2,r+e_2\}$ belong to the same bad block.
    If $e_2$ is the sum of two vectors in $B_H$, say $e_2=p+q$, then the two facets $\{u,p,q,r\}$ and $\{u,p+u,q+u,r+u\}$ belong to the same bad block, and the second facet is again of the possible form.
    We do not group the remaining case $e_2=p+q+r$.
    Since there are $2$ choices of $u\in V+e_1$, the possible facets meet at most
    $$
        \frac{2\cdot4\cdot12}{2} +\frac{2\cdot4\cdot12}{2} +2\cdot4\cdot4  =128<168
    $$
    blocks.

    \noindent\textsf{\textbf{Case 3: $d=2$.}}
    In this case $V=\Span\{e_2,e_3\}$.
    There are two hyperplanes $H$ containing $V$ and not containing $e_1$.
    Fix such an $H$.
    A direct count shows that the numbers of unordered bases $B_H=\{p,q,r\}$ satisfying
    $$
        |B_H\cap V|=2, \quad |B_H\cap V|=1, \quad \text{ and } \quad  |B_H\cap V|=0
    $$
    are $12$, $12$, and $4$, respectively.

    If $|B_H\cap V|=2$, say $p,q\in V$, then the three facets
    $$
        \{u,p,q,r\}, \quad  \{u+p,p,q+p,r+p\}, \quad \text{ and } \quad \{u+q,p+q,q,r+q\}
    $$
    belong to the same bad block.
    If $|B_H\cap V|=1$, say $p\in V$, then the two facets~$\{u,p,q,r\}$ and~$\{u+p,p,q+p,r+p\}$ belong to the same bad block.
    We do not group the case $|B_H\cap V|=0$.
    Since there are $4$ choices of $u\in V+e_1$, the possible facets meet at most
    $$
        \frac{4\cdot2\cdot12}{3} +\frac{4\cdot2\cdot12}{2} +4\cdot2\cdot4 =112<168
    $$
    blocks.

    \noindent\textsf{\textbf{Case 4: $d=3$.}}
    In this case $V=\Span\{e_2,e_3,e_4\}$.
    Then $H=V$ is the only hyperplane containing $V$.
    For every $u\in V+e_1$ and every basis $\{p,q,r\}$ of $V$, the four facets $\{u,p,q,r\}$, $\{u+p,p,q+p,r+p\}$, $\{u+q,p+q,q,r+q\}$, and $\{u+r,p+r,q+r,r\}$ belong to the same bad block.
    Therefore the possible facets meet at most
    $$
        \frac{8\cdot28}{4}=56<168
    $$
    blocks.

    In every case, $\Sigma(\lambda^\R)$ meets fewer than $168$ blocks of the bad-block partition.
    By Lemma~\ref{lem:block-avoidance}, $\lambda^\R$ admits a lift.
\end{proof}

\begin{theorem}\label{thm:join-type-3-spheres}
    Let $K$ and $L$ be simplicial complexes with nonempty vertex sets.
    If $K\ast L$ is a simplicial $3$-sphere, then $K\ast L$ has the toric lifting property.
\end{theorem}
\begin{proof}
    Since $K\ast L$ is pure $3$-dimensional, the complexes $K$ and $L$ are pure and, up to interchanging the factors, their dimensions are either $(1,1)$ or $(0,2)$. 
    Therefore the assertion follows from Lemmas~\ref{lem:join-1-1} and~\ref{lem:join-0-2}.
\end{proof}

\section{Non-liftability of universal complex} \label{non-liftability}

We keep all notation from the previous sections. 
In this section, we show that Lemma~\ref{lem:qle13} is best possible with respect to the cardinality of $S$. 
More precisely, we show that if $|S|\geq 14$, then the inclusion $X[S]\hookrightarrow X(\Z_2^4)$ does not admit a lift.

We use the following finite congruence obstruction. 
The algorithm is only an obstruction test: the output $\UNSAT$ proves non-liftability, while the output $\SAT$ is inconclusive for integral liftability.

Let $K$ be a finite pure $3$-dimensional complex on $[m]$, let $\lambda^\R \colon [m]\to \Z_2^4$ be a mod~$2$ characteristic map, and let $N$ be an even positive integer.
Choose an ordered facet $\tau_0=(p_1,p_2,p_3,p_4)\in \cF(K)$, and put
$$
    M_0=(\lambda^\R(p_1),\lambda^\R(p_2),\lambda^\R(p_3),\lambda^\R(p_4)) \in \GL(4,\Z_2).
$$
Set $\mu^\R(v)\coloneq M_0^{-1}\lambda^\R(v)$ for $v\in \cV(K)$.

For $v\notin\tau_0$, define
$$
    C_N(v)\coloneq \left\{ u\in \Z_N^4  \middle| 
        \begin{array}{l}
            u\equiv \mu^\R(v)\pmod2,\\
            \det(u,e_i,e_j,e_k)\equiv \pm1 (\bmod N) \text{ for }\{v,p_i,p_j,p_k\}\in \cF(K)
        \end{array}
    \right\}.
$$
Let $D_N(v)$ be a set of representatives of the sign orbits $C_N(v)/\{\pm1\}$.
For the base vertices, set $D_N(p_i)\coloneq\{e_i\}$ for $i=1,2,3,4$.

Given a partial assignment of vectors to some vertices of $K$, we call it \emph{admissible} if every facet $\sigma=\{z_1,z_2,z_3,z_4\}\in \cF(K)$ whose four vertices have already been assigned satisfies
$$
    \det(u_{z_1},u_{z_2},u_{z_3},u_{z_4}) \equiv \pm1\pmod N,
$$
where $u_{z_i}\in D_N(z_i)$ for $i = 1, 2, 3, 4$.
\begin{algorithm}[t]
\caption{Finite congruence obstruction test} \label{alg:finite-congruence-obstruction}
\begin{algorithmic}[1]
    \Require A finite pure $3$-complex $K$ on $[m]$, a mod~$2$ characteristic map
    $\lambda^\R \colon [m] \to\Z_2^4$ over $K$, and an even integer $N$.
    \Ensure $\SAT$ or $\UNSAT$.
    
    \State Choose an ordered facet $\tau_0=(p_1,p_2,p_3,p_4)\in \cF(K)$.
    \State Construct the domains $D_N(v)$ as above.
    \State Choose an order $w_1,\ldots,w_r$ of $[m] \setminus\tau_0$.
    \State $\cP\gets\{\phi_0\}$, where $\phi_0(p_i)=e_i$ for $i=1,2,3,4$.
    
    \For{$k=1,\ldots,r$}
        \State
        $$
            \cP\gets  \left\{ \phi\cup\{w_k\mapsto u\} \middle| \phi\in\cP, u\in D_N(w_k), \phi\cup\{w_k\mapsto u\}\text{ is admissible} \right\}.
        $$
        \If{$\cP=\varnothing$}
            \State \Return $\UNSAT$.
        \EndIf
    \EndFor
    
    \State \Return $\mathrm{SAT}$.
\end{algorithmic}
\end{algorithm}

\begin{lemma}\label{lem:finite-congruence-obstruction}
    If Algorithm~\ref{alg:finite-congruence-obstruction} returns $\mathrm{UNSAT}$ for the input $(K,\lambda^\R,N)$, then $\lambda^\R$ admits no integral lift.
\end{lemma}
\begin{proof}
    Suppose that $\lambda$ is an integral lift of $\lambda^\R$. 
    Let $\tau_0=(p_1,p_2,p_3,p_4)$ be the ordered facet chosen by the algorithm, and set $M_\lambda=(\lambda(p_1),\lambda(p_2),\lambda(p_3),\lambda(p_4))$.
    Since $\tau_0$ is a facet, $M_\lambda\in GL(4,\Z)$. 
    Hence $\lambda_{\tau_0}(v)\coloneq M_\lambda^{-1}\lambda(v)$ is a normalized integral lift satisfying $\lambda_{\tau_0}(p_i)=e_i$ for $i=1,2,3,4$.
    Moreover, $\lambda_{\tau_0}(v)\equiv M_0^{-1}\lambda^\R(v)=\mu^\R(v)\pmod2$.
    
    For each $v\notin\tau_0$, the reduction of $\lambda_{\tau_0}(v)$ modulo $N$ lies in $C_N(v)$, because the defining determinant conditions for $C_N(v)$ are necessary facet conditions. 
    Replacing a non-base column by its negative preserves the mod $2$ reduction and changes every determinant only by a sign. 
    Therefore, after choosing the sign representatives used in $D_N(v)$, the normalized lift gives an assignment $u_v\in D_N(v)$ for $v\in V(K)$.
    For every facet $\sigma=\{z_1,z_2,z_3,z_4\}\in \cF(K)$, the corresponding integral determinant is $\pm1$. 
    Hence $\det(u_{z_1},u_{z_2},u_{z_3},u_{z_4})\equiv\pm1\pmod N$.
    Thus this assignment survives the search, so the algorithm would return $\SAT$, contradicting the assumption.
    Therefore no integral lift exists.
\end{proof}

The non-liftability of~$X(\Z_2^4)$ was originally proved by Ayzenberg~\cite{Ayzenberg2016} and Sun~\cite{Sun2017}.
Using Algorithm~\ref{alg:finite-congruence-obstruction}, we have a stronger proposition.

\begin{proposition}\label{prop:14-subset-nonliftable}
    Let $S\subset \cV$ with $|S|\geq14$. 
    Then the inclusion $X[S]\hookrightarrow X(\Z_2^4)$ does not admit a lift to a nondegenerate simplicial map~$X[S]\to X(\Z^4)$.
\end{proposition}
\begin{proof}
    It is enough to consider $S=\cV\setminus\{a_{15}\}$.
    Indeed, $\GL(4,\Z_2)$ acts transitively on the nonzero vectors of $\Z_2^4$, and hence transitively on the $14$-element subsets of~$\cV$, since each such subset is the complement of one nonzero vector. 
    Moreover, any element of $\GL(4,\Z_2)$ has an integral lift in $\GL(4,\Z)$, so liftability is preserved under this action.
    The case $|S|=15$ follows by restricting a hypothetical lift of~$X[\cV]$ to any $14$-vertex induced subcomplex.
    
    We apply Algorithm~\ref{alg:finite-congruence-obstruction} to
    $$
        K=X[S],\quad \lambda^\R=\iota \colon S\hookrightarrow\Z_2^4, \quad \text{ and } \quad N=8.
    $$
    Choose the ordered facet $\tau_0=(a_8, a_4, a_2, a_1)$, and set $W\coloneq S\setminus \tau_0$.
    For $v=(v_1,v_2,v_3,v_4)\in W$, put
    $$
        \supp(v)\coloneq\{i\mid v_i=1\} \quad \text{ and } \quad m(v)\coloneq\min\supp(v).
    $$
    In this special case, Algorithm~\ref{alg:finite-congruence-obstruction} produces
    the following domains
    $$
        D_v= \left\{ u=(u_1,u_2,u_3,u_4)\in \Z_8^4 \middle|
            \begin{array}{ll}
                u_i\in\{1,7\}, & \text{if }v_i=1,\\
                u_i\in\{0,2,4,6\}, & \text{if }v_i=0,\\
                u_{m(v)}=1
            \end{array}
        \right\}
    $$
    for each $v\in W$.
    For the four base vertices, the domains are
    $$
        D_{a_8}=\{e_1\},\quad D_{a_4}=\{e_2\},\quad D_{a_2}=\{e_3\},\quad \text{ and } \quad D_{a_1}=\{e_4\}.
    $$
    
    Let us explain the above description of $D_v$. 
    If $v_i=1$, then $v$, together with the three standard basis vectors $e_j$ for $j\neq i$, forms a facet of~$X[S]$. 
    Hence the determinant condition for this facet forces the $i$th coordinate of any normalized lift to be congruent to $\pm1$ modulo $8$. 
    If $v_i=0$, the parity condition gives an even residue modulo $8$. 
    Finally, the condition $u_{m(v)}=1$ chooses one representative from the sign orbit~$\{u,-u\}$.
    
    The set $W$ consists of six vertices of support size $2$ and four vertices of support size $3$, and
    $$
        |D_v|= \begin{cases}
            32, & |\supp(v)|=2,\\
            16, & |\supp(v)|=3.
        \end{cases}
    $$
    Thus the normalized mod $8$ search space has size $32^6\cdot16^4=2^{46}$ before imposing the remaining facet determinant conditions.
    
    Since $X(\Z_2^4)$ has $840$ facets and the number of facets containing $a_{15}$ is $\frac{14\cdot12\cdot8}{3!}=224$, we have $|\cF(X[S])|=840-224=616$.
    
    We run Algorithm~\ref{alg:finite-congruence-obstruction} with the vertex order
    $$
        a_3, a_5, a_6, a_7, a_9, a_{10}, a_{11}, a_{12}, a_{13},\text{ and } a_{14}.
    $$
    All arithmetic is performed in $\Z_8$. 
    The numbers of surviving partial assignments after each step are as follows:
    $$
        \begin{array}{c|cccccccccc}
            \text{last assigned vertex} & a_3 & a_5 & a_6 & a_7 & a_9 & a_{10} & a_{11} & a_{12} & a_{13} & a_{14} \\ \hline
            \text{surviving assignments} & 32 & 192 & 1280 & 7168 & 352 & 256 & 192 & 96 & 32 & 0 .
        \end{array}
    $$
    Thus Algorithm~\ref{alg:finite-congruence-obstruction} returns $\UNSAT$.
    
    By Lemma~\ref{lem:finite-congruence-obstruction}, the inclusion~$X[\cV\setminus\{a_{15}\}] \hookrightarrow X(\Z_2^4)$ does not admit an integral lift, as desired.
\end{proof}

Set $X = X(\Z_2^4)$, and fix the standard facet $\sigma_0\coloneq\{a_1,a_2,a_4,a_8\}\in \cF(X)$.

\begin{lemma} \label{lem:any_pair_appear}
    For every facet $\sigma\in \cF(X)$, there exist 
    \begin{itemize} 
        \item[(1)] a simplicial 3-sphere $K_\sigma$,
        \item[(2)] a nondegenerate simplicial map $\widehat{{\lambda}_\sigma^\R}: K_\sigma\to X$,
        \item[(3)] two distinct facets $\tau^0_\sigma, \tau^1_\sigma\in \cF(K_\sigma)$,
    \end{itemize} 
    such that 
    $$  
        \widehat{{\lambda}_\sigma^\R}(\tau_\sigma^0)=\sigma_0 \quad \text{ and } \quad  \widehat{{\lambda}_\sigma^\R}(\tau_\sigma^1)=\sigma.
    $$  
\end{lemma}
\begin{proof}
    If $\sigma=\sigma_0$, the assertion is immediate. 
    Thus assume $\sigma\ne\sigma_0$.    
    
    We first prove the statement for a facet $\sigma\in \cF(X)$ satisfying $|\sigma\cap \sigma_0| = 3$. 
    Without loss of generality, assume that $\sigma = \{a_1, a_2, a_4, v\}$. 
    Let 
    $$
        K_\sigma = \partial\Delta_1^1 \ast \partial\Delta_2^1 \ast \partial\Delta_3^1 \ast \partial\Delta_4^1,
    $$ where $\Delta_i^1$ is a copy of the 1-simplex $\Delta^1$ for $i = 1, 2, 3, 4$. 
    Color the two vertices of~$\partial \Delta_1^1$ by~$a_1$, the two vertices of~$\partial \Delta_2^1$ by~$a_2$, the two vertices of~$\partial\Delta_3^1$ by~$a_4$, and the two vertices of~$\partial \Delta_4^1$ by~$a_8$ and~$v$.
    The coloring of vertices of any facet of~$K_\sigma$ is either $a_1, a_2, a_4, a_8$, or $a_1, a_2, a_4, v$. 
    So this coloring defines a nondegenerate simplicial map~$\widehat{{\lambda}_\sigma^\R} \colon K_\sigma\to X$ satisfying (1)--(3).
    
    Then for an arbitrary facet $\sigma\in \cF(X)$, since the basis graph of $\Z_2^4$ is connected, there exists a sequence of facets~$\sigma_0, \sigma_1, \dots, \sigma_n = \sigma\in \cF(X)$ such that $|\sigma_i\cap \sigma_{i+1}| = 3$ for all $0\leq i\leq n-1$. 
    By the above argument, for each $i$, there exists a simplicial 3-sphere $K_i$, a nondegenerate simplicial map $\widehat{{\lambda}_i^\R}\colon K_i\to X$, and two distinct facets~$\tau^0_i, \tau^1_i\in \cF(K_i)$ such that
    $$
        \widehat{\lambda_i^{\R}}(\tau^0_i) = \sigma_i \quad \text{ and }\quad  \widehat{\lambda_i^{\R}}(\tau^1_i) = \sigma_{i+1}.
    $$
    Since $\widehat{\lambda_0^{\R}}(\tau_0^1) = \sigma_1 = \widehat{\lambda_1^{\R}}(\tau_1^0)$ and $\tau_0^1$ and $\tau_1^0$ are simplices, there exists an isomorphism $\varphi_0 \colon \tau_0^1\to \tau_1^0$ such that $\widehat{\lambda_0^{\R}}|_{\tau_0^1} = \widehat{\lambda_1^{\R}}|_{\tau_1^0}\circ \varphi_0$. 
    Let $K_1' = K_0\#_{\varphi_0} K_1$ be the connected sum of $K_0$ and $K_1$ along the facets $\tau_0^1$ and $\tau_1^0$ via $\varphi_0$.
    Then we obtain the following properties.
    \begin{itemize}
        \item $K_1'$ is a simplicial 3-sphere.
        \item Since $\widehat{\lambda_0^{\R}}|_{\tau_0^1} =  \widehat{\lambda_1^{\R}}\circ \varphi_0$, the two maps~$\widehat{\lambda_0^{\R}}$ and~$\widehat{\lambda_1^{\R}}$ induce a well-defined nondegenerate simplicial map $\widehat{\lambda^{\R}} \colon K_1'\to X$ such that $\widehat{\lambda^{\R}}|_{K_0\setminus \tau_0^1} = \widehat{\lambda_0^{\R}}$ and $\widehat{\lambda^{\R}}|_{K_1\setminus \tau_1^0} = \widehat{\lambda_1^{\R}}$. 
        \item $\tau_0^0, \tau_1^1$ are distinct facets of $K_1'$ with $\widehat{\lambda^{\R}}(\tau_0^0) = \sigma_0$ and $\widehat{\lambda^{\R}}(\tau_1^1) = \sigma_2$.
    \end{itemize}
    Iterating this process yields a simplicial 3-sphere and a nondegenerate simplicial map satisfying (1)--(3).
\end{proof}

\begin{proposition}\label{prop:surjective-image} 
    There exist a simplicial $3$-sphere $K$ and a nondegenerate simplicial map~$\widehat{\lambda^{\R}}\colon K\to X$ which is surjective on facets. 
\end{proposition} 
\begin{proof} 
    Write $\cF(X)=\{\sigma_0,\sigma_1,\dots,\sigma_r\}$, where $r=839$.
    
    We first construct a simplicial $3$-sphere~$K_0$ together with a nondegenerate simplicial map $\widehat{\lambda^{\R}_0}\colon K_0\to X$ such that every facet of $K_0$ maps onto $\sigma_0$, and $K_0$ has at least $r$ pairwise distinct facets. 
    Let $K_0\coloneq\partial\Delta^1 \ast \partial\Delta^1 \ast C_{2r}$, where $C_{2r}$ is a cycle on $2r$ vertices.
    Color the two vertices of the first copy of $\partial\Delta^1$ by $a_1$, the two vertices of the second copy by $a_2$, and the vertices of the cycle alternately by $a_4$ and $a_8$. 
    This defines a nondegenerate simplicial map $\widehat{\lambda^{\R}_0}\colon K_0\to X$, and every facet of $K_0$ maps to $\sigma_0$. 
    Moreover, $f_3(K_0)=2\cdot 2\cdot (2r)=8r$, so $K_0$ contains at least $r$ pairwise distinct facets. 
    
    Choose $r$ distinct facets $\rho_1,\dots,\rho_r\in \cF(K_0)$.
    We now construct inductively simplicial $3$-spheres $K^{(0)},K^{(1)},\dots,K^{(r)}$ and nondegenerate simplicial maps 
    $$
        \widehat{\lambda^{\R}_{(j)}}\colon K^{(j)}\to X
    $$ 
    such that $\rho_i$ is still a facet of $K^{(j)}$, $\widehat{\lambda^{\R}_{(j)}}(\rho_i) = \sigma_0$ for all $i = j+1, \dots, r$, and $\{\sigma_0,\sigma_1,\dots,\sigma_j\}\subset \operatorname{im}(\widehat{\lambda^{\R}_{(j)}})$.
    Set $K^{(0)}\coloneq K_0$ and $\widehat{\lambda^{\R}_{(0)}} \coloneq \widehat{\lambda^{\R}_0}$.
    
    Suppose that $K^{(j-1)}$ and $\widehat{\lambda^{\R}_{(j-1)}}$ have been constructed. 
    By Lemma \ref{lem:any_pair_appear}, there exist a simplicial 3-sphere $K_{\sigma_j}$, a nondegenerate simplicial map $\widehat{\lambda^{\R}_{\sigma_j}}: K_{\sigma_j}\to X$, and two distinct facets $\tau_{\sigma_j}^0$, $\tau_{\sigma_j}^1$ of $K_{\sigma_j}$ such that $\widehat{\lambda^{\R}_{\sigma_j}}(\tau_{\sigma_j}^0) = \sigma_0$ and $\widehat{\lambda^{\R}_{\sigma_j}}(\tau_{\sigma_j}^1) = \sigma_j$.
    Since $\widehat{\lambda^{\R}_{(j-1)}}(\rho_j)=\sigma_0= \widehat{\lambda^{\R}_{\sigma_j}}(\tau_{\sigma_j}^0)$, there exists a unique color-preserving simplicial isomorphism $\phi_j\colon \rho_j\to \tau_{\sigma_j}^0$ such that $\widehat{\lambda^{\R}_{(j-1)}}\big|_{\rho_j} = \widehat{\lambda^{\R}_{\sigma_j}}\circ \phi_j$.
    Now take the simplicial connected sum $K^{(j)}\coloneq K^{(j-1)}\#_{\phi_j} K_{\sigma_j}$ along the facets~$\rho_j$ and~$\tau_{\sigma_j}^0$. 
    Since connected sum of simplicial $3$-spheres along facets is again a simplicial $3$-sphere, $K^{(j)}$ is a simplicial $3$-sphere. 
    Moreover, the maps $ \widehat{\lambda^{\R}_{(j-1)}}$ and $\widehat{\lambda^{\R}_{\sigma_j}}$ glue to a nondegenerate simplicial map $\widehat{\lambda^{\R}_{(j)}}$. 
    By construction, the facet~$\tau_{\sigma_j}^1$ survives the connected sum, so $\sigma_j$ is contained in the image of $\widehat{\lambda^{\R}_{(j)}}$. 
    Also, all previously realized facets remain in the image. 
    Thus 
    $$
        \{\sigma_0,\sigma_1,\dots,\sigma_j\}\subset \operatorname{im}(\widehat{\lambda^{\R}_{(j)}}).
    $$
    After $r$ steps, we obtain a simplicial $3$-sphere $K\coloneq K^{(r)}$ and a nondegenerate simplicial map~$\widehat{\lambda^{\R}} \coloneq \widehat{\lambda^{\R}_{(r)}}\colon K\to X$, as desired.
\end{proof}

Proposition~\ref{prop:14-subset-nonliftable} shows that the sufficient criterion obtained by lifting the inclusion~$X[S] \hookrightarrow X(\Z_2^4)$ cannot cover all image subcomplexes of size at least $14$, while Proposition~\ref{prop:surjective-image} shows that the full universal complex is visible from simplicial $3$-spheres.

\providecommand{\bysame}{\leavevmode\hbox to3em{\hrulefill}\thinspace}
\providecommand{\MR}{\relax\ifhmode\unskip\space\fi MR }
\providecommand{\MRhref}[2]{%
  \href{http://www.ams.org/mathscinet-getitem?mr=#1}{#2}
}
\providecommand{\href}[2]{#2}

\end{document}